\documentclass[12pt]{amsart}
\usepackage{amssymb}
\usepackage{verbatim}
\usepackage{hyperref}
\usepackage{color}

\begin{document}

\newtheorem{theorem}{Theorem}    
\newtheorem{proposition}[theorem]{Proposition}
\newtheorem{conjecture}[theorem]{Conjecture}
\def\theconjecture{\unskip}
\newtheorem{corollary}[theorem]{Corollary}
\newtheorem{lemma}[theorem]{Lemma}
\newtheorem{sublemma}[theorem]{Sublemma}
\newtheorem{observation}[theorem]{Observation}
\theoremstyle{definition}
\newtheorem{definition}{Definition}
\newtheorem{notation}[definition]{Notation}
\newtheorem{remark}[definition]{Remark}
\newtheorem{question}[definition]{Question}
\newtheorem{questions}[definition]{Questions}
\newtheorem{example}[definition]{Example}
\newtheorem{problem}[definition]{Problem}
\newtheorem{exercise}[definition]{Exercise}

\numberwithin{theorem}{section} \numberwithin{definition}{section}
\numberwithin{equation}{section}

\def\earrow{{\mathbf e}}
\def\rarrow{{\mathbf r}}
\def\uarrow{{\mathbf u}}
\def\varrow{{\mathbf V}}
\def\tpar{T_{\rm par}}
\def\apar{A_{\rm par}}

\def\reals{{\mathbb R}}
\def\torus{{\mathbb T}}
\def\heis{{\mathbb H}}
\def\integers{{\mathbb Z}}
\def\naturals{{\mathbb N}}
\def\complex{{\mathbb C}\/}
\def\distance{\operatorname{distance}\,}
\def\support{\operatorname{support}\,}
\def\dist{\operatorname{dist}\,}
\def\Span{\operatorname{span}\,}
\def\degree{\operatorname{degree}\,}
\def\kernel{\operatorname{kernel}\,}
\def\dim{\operatorname{dim}\,}
\def\codim{\operatorname{codim}}
\def\trace{\operatorname{trace\,}}
\def\Span{\operatorname{span}\,}
\def\dimension{\operatorname{dimension}\,}
\def\codimension{\operatorname{codimension}\,}
\def\nullspace{\scriptk}
\def\kernel{\operatorname{Ker}}
\def\ZZ{ {\mathbb Z} }
\def\p{\partial}
\def\rp{{ ^{-1} }}
\def\Re{\operatorname{Re\,} }
\def\Im{\operatorname{Im\,} }
\def\ov{\overline}
\def\eps{\varepsilon}
\def\lt{L^2}
\def\diver{\operatorname{div}}
\def\curl{\operatorname{curl}}
\def\etta{\eta}
\newcommand{\norm}[1]{ \|  #1 \|}
\def\expect{\mathbb E}
\def\bull{$\bullet$\ }
\def\C{\mathbb{C}}
\def\R{\mathbb{R}}
\def\Rn{{\mathbb{R}^n}}
\def\Sn{{{S}^{n-1}}}
\def\M{\mathbb{M}}
\def\N{\mathbb{N}}
\def\Q{{\mathbb{Q}}}
\def\Z{\mathbb{Z}}
\def\F{\mathcal{F}}
\def\L{\mathcal{L}}
\def\S{\mathcal{S}}
\def\supp{\operatorname{supp}}
\def\dist{\operatorname{dist}}
\def\essi{\operatornamewithlimits{ess\,inf}}
\def\esss{\operatornamewithlimits{ess\,sup}}
\def\xone{x_1}
\def\xtwo{x_2}
\def\xq{x_2+x_1^2}
\def\gfz{\genfrac{}{}{0pt}{}}
\newcommand{\abr}[1]{ \langle  #1 \rangle}

\newcommand{\Norm}[1]{ \left\|  #1 \right\| }
\newcommand{\set}[1]{ \left\{ #1 \right\} }
\def\one{\mathbf 1}
\def\whole{\mathbf V}
\newcommand{\modulo}[2]{[#1]_{#2}}

\def\scriptf{{\mathcal F}}
\def\scriptg{{\mathcal G}}
\def\scriptm{{\mathcal M}}
\def\scriptb{{\mathcal B}}
\def\scriptc{{\mathcal C}}
\def\scriptt{{\mathcal T}}
\def\scripti{{\mathcal I}}
\def\scripte{{\mathcal E}}
\def\scriptv{{\mathcal V}}
\def\scriptw{{\mathcal W}}
\def\scriptu{{\mathcal U}}
\def\scriptS{{\mathcal S}}
\def\scripta{{\mathcal A}}
\def\scriptr{{\mathcal R}}
\def\scripto{{\mathcal O}}
\def\scripth{{\mathcal H}}
\def\scriptd{{\mathcal D}}
\def\scriptl{{\mathcal L}}
\def\scriptn{{\mathcal N}}
\def\scriptp{{\mathcal P}}
\def\scriptk{{\mathcal K}}
\def\frakv{{\mathfrak V}}

\title[Compactness of maximal commutators of bilinear operators]
{Compactness of maximal commutators of bilinear Calder\'{o}n-Zygmund singular integral operators}

\author[Yong Ding]{Yong Ding} 

\author[Ting Mei]{Ting Mei}

\author[Qingying Xue]{Qingying Xue}

%

%
\thanks {The first author is supported by NSFC (No.11371057) and  SRFDP (No.20130003110003), the third author is supported by NCET-13-0065.}

\address{Yong Ding\\School of Mathematical
Sciences\\Beijing Normal University\\Laboratory of Mathematics and
Complex Systems\\Ministry of Education\\Beijing, 100875\\
P. R. China}
\email{dingy@bnu.edu.cn}

\address{Ting Mei\\School of Mathematical
Sciences\\Beijing Normal University\\Laboratory of Mathematics and
Complex Systems\\Ministry of Education\\Beijing, 100875\\
P. R. China}
\email{meiting@mail.bnu.edu.cn}

\address{Qingying Xue\\School of Mathematical
Sciences\\Beijing Normal University\\Laboratory of Mathematics and
Complex Systems\\Ministry of Education\\Beijing, 100875\\
P. R. China
}
\email{qyxue@bnu.edu.cn}
\maketitle

\begin{abstract}
Let $T$ be a bilinear Calder\'{o}n-Zygmund singular integral operator and $T_*$ be its corresponding truncated maximal operator. The commutators in the $i$-$th$ entry and the iterated commutators of $T_*$ are defined by
$$
T_{\ast,b,1}(f,g)(x)=\sup_{\delta>0}\bigg|\iint_{|x-y|+|x-z|>\delta}K(x,y,z)(b(y)-b(x))f(y)g(z)dydz\bigg|,
$$
$$
T_{\ast,b,2}(f,g)(x)=\sup_{\delta>0}\bigg|\iint_{|x-y|+|x-z|>\delta}K(x,y,z)(b(z)-b(x))f(y)g(z)dydz\bigg|,
$$
\begin{align*}
T_{\ast,(b_1,b_2)}(f,g)(x)=\sup\limits_{\delta>0}\bigg|\iint_{|x-y|+|x-z|>\delta} K(x,y,z)(b_1(y)-b_1(x))(b_2(z)-b_2(x))f(y)g(z)dydz\bigg|.
\end{align*}
In this paper, the compactness of the commutators $T_{\ast,b,1}$, $T_{\ast,b,2}$ and $T_{\ast,(b_1,b_2)}$ on $L^r(\R^n)$ is established.
\end{abstract}
\section{Introduction and Main Results} 
\setcounter{equation}{0}
Let $T_\Omega$ be the well-known Calder\'{o}n-Zygmund singular integral operator defined by $$T_\Omega f(x)=p.v.\int_{\R^n} \frac{\Omega(x-y)}{|x-y|^n}f(y)\;dy.$$ In 1976, Coifman, Rochberg and Weiss \cite{CRW} defined the following well-known commutator of $T_\Omega$ for smooth functions,
\begin{equation}\label{T1}
[b,T_\Omega]f(x)=b(x)T_\Omega(f)(x)-T_\Omega(bf)(x)=p.v.\int_{
\mathbb{R}^n}(b(x)-b(y)) \frac{\Omega(x-y)}{|x-y|^n}
f(y)dy.
\end{equation}
The authors of \cite{CRW}
proved that $[b,T_\Omega]$ is  bounded on $L^p$ for $1<p<\infty$ when $b\in BMO$ and $\Omega$ satisfies:\\
\text{(i)} $\Omega(\lambda x)=\Omega(x)$, for any $\lambda>0$ and $x\neq 0$,\\
\text{(ii)} $\int_{S^{n-1}}\Omega(x^\prime)\;d\sigma(x^\prime)=0$,\\
\text{(iii)} $|\Omega(x^\prime)-\Omega(y^\prime)|\leq |x^\prime-y^\prime|$, for any $x^\prime,\,y^\prime\in S^{n-1}$.\\
The following characterization of $L^p$-compactness of $[b,T_\Omega]$ was given by Uchiyama in 1978.

\noindent\textbf{Theorem A} (\cite{AU}) Let $\Omega$ satisfy \text{(i)}-\text{(iii)}. If $b\in \cup_{q>1}L^q_{loc}(\R^n)$, then $[b,T_\Omega]$ is a compact operator on $L^p(\R^n)$ for $1<p<\infty$ if and only if $b\in CMO$ (the closure in $BMO(\R^n)$ of $C_c^\infty(\R^n)$).

In 1993, Beatrous and Li \cite{BL} gave the boundedness and the compactness characterizations for $[b,T_\Omega]$ in $L^p$ space and some spaces of homogeneous type. In 2001, Krantz and Li applied the compactness characterization of the commutator $[b,T_\Omega]$ to study Hankel type operators on Bergman spaces (see \cite{KL1},\cite{KL2}). As for the compactness of commutators for the other type operators, such as the Riesz potential, singular integral with variable kernel, parabolic singular integral, Littlewood-Paley operators, one may see \cite{WS} and the recent works \cite{CD1}-\cite{CDW3}.

Let $T$ be a bilinear Calder\'{o}n-Zygmund operator (see \cite{Gra1}) and assume that the kernel $K$ satisfies the usual conditions in such a theory, that is $K\in 2$-$CZK(A,\gamma)$. Let $T_*$ be the corresponding bilinear maximal singular integral operator defined by
\begin{align}\label{T*}
T_{\ast}(f,g)(x)=\sup_{\delta>0}\bigg|\iint_{|x-y|+|x-z|>\delta}K(x,y,z)f(y)g(z)\;dydz\bigg|.
\end{align}
In 2002, Grafakos and Torres \cite{GT1} obtained the following $L^p$-estimate of $T_*$.
\begin{align}\label{bdd of T_*}
\|T_*(f,g)\|_r\leq C\|f\|_p\|g\|_q,
\end{align}for $1<p,q<\infty$, $1/2<r<\infty$ with $\frac{1}{r}=\frac{1}{p}+\frac{1}{q}$.

Let $b,\,b_1,\,b_2\in BMO(\R^n)$. We are interested in the following three maximal commutators of bilinear Calder\'{o}n-Zygmund operators:
\begin{align*}
T_{\ast,b,1}(f,g)(x)=\sup_{\delta>0}|[T_\delta,b]_1(f,g)(x)|=\sup_{\delta>0}|(T_\delta(bf,g)-bT_\delta(f,g))(x)|,
\end{align*}
\begin{align*}
T_{\ast,b,2}(f,g)(x)=\sup_{\delta>0}|[T_\delta,b]_2(f,g)(x)|=\sup_{\delta>0}|(T_\delta(f,bg)-bT_\delta(f,g))(x)|,
\end{align*}
\begin{align*}
T_{\ast,(b_1,b_2)}(f,g)(x)=\sup_{\delta>0}|[[T_\delta,b_1]_1,b_2]_2(f,g)(x)|
=\sup_{\delta>0}|([T_\delta,b_1]_1(f,b_2g)-b_2[T_\delta,b_1]_1(f,g))(x)|,
\end{align*}where $T_\delta(f,g)(x)=\int\int_{|x-y|+|x-z|>\delta} K(x,y,z)f(y)g(z)\;dydz$.\\
Formally, they can take the form
\begin{align}\label{T*b 1}
T_{\ast,b,1}(f,g)(x)=\sup_{\delta>0}\bigg|\iint_{|x-y|+|x-z|>\delta} K(x,y,z)(b(y)-b(x))f(y)g(z)\;dydz\bigg|,
\end{align}
\begin{align}\label{T*b 2}
T_{\ast,b,2}(f,g)(x)=\sup_{\delta>0}\bigg|\iint_{|x-y|+|x-z|>\delta} K(x,y,z)(b(z)-b(x))f(y)g(z)\;dydz\bigg|,
\end{align}
\begin{align}\label{commutators of bilinear maximal singular integral operators}
&T_{\ast,(b_1,b_2)}(f,g)(x)\\
=&\sup_{\delta>0}\bigg|\iint_{|x-y|+|x-z|>\delta} K(x,y,z)(b_1(y)-b_1(x))(b_2(z)-b_2(x))f(y)g(z)\;dydz\bigg|\nonumber.
\end{align}
By the results in \cite{Xue}, the third operator maps $L^p(\R^n)\times L^q(\R^n)\rightarrow L^r(\R^n)$ with $\frac{1}{r}=\frac{1}{p}+\frac{1}{q}$ for all $\frac{1}{2}<r<\infty$, $1<p,\,q<\infty$, with the following estimate:
\begin{align}\label{bdd of T_{*,(b_1,b_2)}}
\|T_{\ast, (b_1,b_2)}(f,g)\|_{L^r(\R^n)}\leq C \prod_{j=1}^2 \|b_j\|_{BMO} \|f\|_{L^p(\R^n)}\|g\|_{L^q(\R^n)}.
\end{align}
\begin{remark}
 We can't find any results for the $L^p$-boundedness of the first two operators, but it is trival. We also can use the ideal in \cite{Xue}. Give the $L^p$-estimates of two maximal commutators controlling $T_{\ast, b,i}$ and obtain the following result.
\begin{align}\label{bdd of T_{*,b,i}}
\|T_{\ast, b,i}(f,g)\|_{L^r(\R^n)}\leq C\|b\|_{BMO} \|f\|_{L^p(\R^n)}\|g\|_{L^q(\R^n)},
\end{align} with $\frac{1}{r}=\frac{1}{p}+\frac{1}{q}$ for all $\frac{1}{2}<r<\infty$, $1<p,\,q<\infty$.
\end{remark}
Compare with the classical compact results in Theorem A, for the compactness of bilinear operators, recently, \'{A}rp\'{a}d B\'{e}nyi and R. H. Torres in \cite{BT0} first studied the compactness for commutators of bilinear Calder\'{o}n-Zygmund singular integral operators. \'{A}rp\'{a}d B\'{e}nyi et al. \cite{BT} also considered compactness properties of commutators of bilinear fractional integrals. Let us recall the definition of the compact bilinear operator (see \cite{BT0}).
\begin{definition}
Let $B_{r,X}=\{x\in X:\|x\|\leq r\}$ be the closed ball of radius $r$ centered at the origin in the normed space $X$. A bilinear operator $T: X\times Y\rightarrow Z$ is called compact if $T(B_{1,X}\times B_{1,Y})$ is precompact in $Z$.
\end{definition}

It is natural to ask whether the compact results still hold for the maximal commutators $T_{*,b,i}$, $T_{*,(b_1,b_2)}$ of the bilinear singular integral operators or not. We have found that there is no result for the compactness of the commutators of $T_{\ast}$ defined in (\ref{T*b 1}), (\ref{T*b 2}) and (\ref{commutators of bilinear maximal singular integral operators}), even in the classical linear case.

The main purpose of the present paper is to show the compactness for the maximal commutators $T_{\ast,b,1}$, $T_{\ast,b,2}$ and $T_{\ast,(b_1,b_2)}$ of bilinear Calder\'{o}n-Zygmund operators when the symbols $b,\,b_1,\,b_2\in CMO(\R^n)$, which denotes the closure of $C_c^\infty(\R^n)$ in the $BMO(\R^n)$ topology. Now, the difficulty lies in that $T_*$ is a sub-linear operator, we can't use the classical known method. The $L^p$-boundedness of the commutators of $T_*$ comes from the $L^p$-boundedness of two maximal commutators which control the commutators of $T_*$ (see \cite{Xue}). In fact, we can get the compact results of the sum of the two maximal commutators controlling the commutators of $T_*$, but from this we can't deduce the compact result for the commutators of $T_*$ (see also \cite{CD2}).

Our main results are as follows.
\begin{theorem}\label{thm compact of maximal singular integral 1}
Let $1\le r<\infty$, $1<p,\,q<\infty$ with $\frac{1}{r}=\frac{1}{p}+\frac{1}{q}$. For any $f\in L^p(\R^n),\,g\in L^q(\R^n)$, let $T_{\ast,b,1}(f,g)$, $T_{\ast,b,2}(f,g)$ be defined in \rm{(\ref{T*b 1})} and \rm{(\ref{T*b 2})}. If $b\in CMO(\R^n)$, then $T_{\ast,b,1}$ and $T_{\ast,b,2}$ are compact operators from $L^{p}(\R^n)\times L^{q}(\R^n)$ to $L^r(\R^n)$.
\end{theorem}

\begin{theorem}\label{thm compact of maximal singular integral 2}
Let $1\le r<\infty$, $1<p,\,q<\infty$ with $\frac{1}{r}=\frac{1}{p}+\frac{1}{q}$. For any $f\in L^p(\R^n),\,g\in L^q(\R^n)$, let $T_{\ast,(b_1,b_2)}(f,g)$ be defined in \rm{(\ref{commutators of bilinear maximal singular integral operators})}. If $b_1,\,b_2\in CMO(\R^n)$, then $T_{\ast, (b_1,b_2)}$ is a compact operator from $L^{p}(\R^n)\times L^{q}(\R^n)$ to $L^r(\R^n)$.
\end{theorem}
\remark Theorem \ref{thm compact of maximal singular integral 1}, \ref{thm compact of maximal singular integral 2} also hold for m-linear maximal Calder\'{o}n-Zygmund singular integral operators (including linear case $m=1$). The essential ideas in the proof are similar, of course, with more complicated and delicate division in the main steps of the proof in Section 2,3. For simplicity, we omit the proof.
\remark The compact results are new even for the linear case which can be proved by using similar ideas and steps as in the proof of Theorem \ref{thm compact of maximal singular integral 1}.

\section{The proof of Theorem \ref{thm compact of maximal singular integral 1}} 
\setcounter{equation}{0}
\noindent
In this part, we will give the proof of  Theorem \ref{thm compact of maximal singular integral 1}. We first give the following lemmas.
\begin{lemma}\label{Frechet Kolmogolov thm} \rm{(Frechet-Kolmogorov)}
\cite{YK}
A subset $G$ of $L^p(\R^n) (1\leq p<\infty)$ is strongly pre-compact if and only if $G$ satisfies the following conditions:\\
$(a)$ $\sup\limits_{f\in G} \|f\|_p <\infty$;\\
$(b)$ $\lim\limits_{\alpha\rightarrow \infty} \|f\chi_{E_\alpha}\|_p=0$, uniformly for $f\in G$, where $E_\alpha=\{x\in\R^n:|x|>\alpha\}$;\\
$(c)$ $\lim\limits_{|h|\rightarrow0} \|f(\cdot+h)-f(\cdot)\|_p=0$, uniformly for $f\in G$.
\end{lemma}

\begin{lemma}
For any $\delta>0$, $0<\epsilon<\frac{1}{2}$, we have the following inequalities
\begin{align}\label{estimate 1}
\iint_{\frac{\delta}{1+2\epsilon}\leq|y|+|z|\leq \delta}\frac{dydz}{(|y|+|z|)^{2n}}\leq C[1-(1+2\epsilon)^{-n}],
\end{align}
\begin{align}\label{estimate 2}
\iint_{\delta\leq|y|+|z|\leq \frac{\delta}{1-2\epsilon}}\frac{dydz}{(|y|+|z|)^{2n}}\leq C[(1-2\epsilon)^{-n}-1],
\end{align}where the constant $C$ is independent of $\delta$ and $\epsilon$.
\end{lemma}
\begin{proof}
We first give the estimate for (\ref{estimate 1}). Simple computation and spherical coordinates transformations give that
\begin{align*}
&\iint_{\frac{\delta}{1+2\epsilon}\leq|y|+|z|\leq \delta}\frac{dydz}{(|y|+|z|)^{2n}}\\
\leq&\int_{|z|\leq\frac{\delta}{1+2\epsilon}}\int_{\frac{\delta}{1+2\epsilon}-|z|\leq|y|\leq \delta-|z|}\frac{dydz}{(|y|+|z|)^{2n}}
+\int_{\frac{\delta}{1+2\epsilon}\leq|z|\leq\delta}\int_{|y|\leq \delta-|z|}\frac{dydz}{(|y|+|z|)^{2n}}\\
\leq&C\int_{|z|\leq\frac{\delta}{1+2\epsilon}}\int_{\frac{\delta}{1+2\epsilon}-|z|}^{\delta-|z|}\frac{r^{n-1}}{(r+|z|)^{2n}}drdz
+C\int_{\frac{\delta}{1+2\epsilon}\leq|z|\leq\delta}\int_0^{\delta-|z|}\frac{r^{n-1}}{(r+|z|)^{2n}}drdz\\
\leq&C\int_{|z|\leq\frac{\delta}{1+2\epsilon}}[\delta^{-n}(1+2\epsilon)^{n}-\delta^{-n}]dz
+C\int_{\frac{\delta}{1+2\epsilon}\leq|z|\leq\delta}(|z|^{-n}-\delta^{-n})dz\\
\leq&C[1-(1+2\epsilon)^{-n}]+C\delta^{-n}[\delta^n-\delta^n(1+2\epsilon)^{-n}]\leq C[1-(1+2\epsilon)^{-n}].
\end{align*}
Analogously, we also can obtain (\ref{estimate 2}).
\begin{align*}
&\iint_{\delta\leq|y|+|z|\leq \frac{\delta}{1-2\epsilon}}\frac{dydz}{(|y|+|z|)^{2n}}\\
\leq&\int_{|z|\leq\delta}\int_{\delta-|z|\leq|y|\leq\frac{\delta}{1-2\epsilon}-|z|}\frac{dydz}{(|y|+|z|)^{2n}}
+\int_{\delta\leq|z|\leq \frac{\delta}{1-2\epsilon}}\int_{|y|\leq \frac{\delta}{1-2\epsilon}-|z|}\frac{dydz}{(|y|+|z|)^{2n}}\\
\leq&C\int_{|z|\leq\delta}[\delta^{-n}-\delta^{-n}(1-2\epsilon)^{n}]dz
+C\int_{\delta\leq|z|\leq\frac{\delta}{1-2\epsilon}}(|z|^{-n}-\delta^{-n}(1-2\epsilon)^n)dz\\
\leq&C[1-(1-2\epsilon)^{n}]+C\delta^{-n}[\delta^n(1-2\epsilon)^{-n}-\delta^n]\leq C[(1-2\epsilon)^{-n}-1].
\end{align*}
We complete the proof of the lemma.
\end{proof}

\noindent
{\textbf{\it Proof of Theorem \ref{thm compact of maximal singular integral 1}.}} We only prove $i=1$. Without loss of generality, let $B_1$, $B_2$ be unit balls in $L^p(\R^n)$ and $L^q(\R^n)$, respectively. We need to show the set $\{T_{*,b,1}(f,g): f\in B_1,\,g\in B_2\}$ is strongly pre-compact in  $L^r(\R^n)$ with $b\in CMO(\R^n)$.

We first show that if the set $\{T_{*,b,1}(f,g): f\in B_1,\,g\in B_2\}$ is strongly pre-compact in $L^r(\R^n)$ for $b\in C_c^\infty(\R^n)$, then the set $\{T_{*,b,1}(f,g): f\in B_1,\,g\in B_2\}$ is also strongly pre-compact in $L^r(\R^n)$ for $b\in CMO$. In fact, suppose that $b\in CMO$, then for any $\epsilon>0$, there exists $b^\epsilon\in C_c^\infty(\R^n)$ such that $\|b-b^\epsilon\|_{BMO}<\epsilon$. It is easy to see that
\begin{align*}
|&T_{*,b,1}(f,g)(x)-T_{*,b^\epsilon,1}(f,g)(x)|\\
\leq&\sup_{\delta>0}\bigg|\iint_{|x-y|+|x-z|>\delta} \bigg[(b(y)-b(x))-(b^\epsilon(y)-b^\epsilon(x)) \bigg]K(x,y,z)f(y)g(z)\;dydz\bigg|\\
\leq &T_{*,b-b^\epsilon,1}(f,g)(x).
\end{align*}
Then combine with the above inequality and (\ref{bdd of T_{*,b,i}}), we have
\begin{align}\label{b^e-b}
\|T_{*,b,1}(f,g)-T_{*,b^\epsilon,1}(f,g)\|_r\leq\|T_{*,b-b^\epsilon,1}(f,g)\|_r\leq C\|b-b^\epsilon\|_{BMO}\|f\|_p\|g\|_q\leq C\epsilon.
\end{align}
 Denote $F_1:=\{T_{*,b^\epsilon,1}(f,g): f\in B_1,\,g\in B_2\}$, then $(a)$, $(b)$ and $(c)$ in Lemma \ref{Frechet Kolmogolov thm} hold for $F_1$. We need to show that $(a)$, $(b)$ and $(c)$ also hold for the set $\widetilde{F}_1:=\{T_{*,b,1}(f,g): f\in B_1,\,g\in B_2\}$. (\ref{b^e-b}) gives that
\begin{align}\label{a0}
\sup_{f\in B_1,g\in B_2}\|T_{*,b,1}(f,g)\|_r\leq\sup_{f\in B_1,g\in B_2}\|T_{*,b^\epsilon,1}(f,g)\|_r+C\epsilon<\infty.
\end{align}On the other hand,
\begin{align}\label{b0}
\lim_{\alpha\rightarrow\infty}\|T_{*,b,1}(f,g)\chi_{E_\alpha}\|_r\leq&\lim_{\alpha\rightarrow\infty}\|T_{*,b^\epsilon,1}(f,g)\chi_{E_\alpha}\|_r+\|T_{*,b-b^\epsilon,1}(f,g)\|_r\\
\leq& C\epsilon\rightarrow 0,\quad (\epsilon\rightarrow0).\nonumber
\end{align}
\begin{align}\label{c0}
&\lim_{|h|\rightarrow0}\|T_{*,b,1}(f,g)(\cdot+h)-T_{*,b,1}(f,g)(\cdot)\|_r\\
\leq&\lim_{|h|\rightarrow0}\|T_{*,b^\epsilon,1}(f,g)(\cdot+h)-T_{*,b^\epsilon,1}(f,g)(\cdot)\|_r+2\|T_{*,b-b^\epsilon,1}(f,g)\|_r\nonumber\\
\leq& C\epsilon\rightarrow 0,\quad (\epsilon\rightarrow0).\nonumber
\end{align}It is obvious to see that the above limits hold uniformly in $\widetilde{F}_1$. Therefore, we know $\widetilde{F}_1$ is strongly pre-compact in $L^r(\R^n)$ for $b\in CMO$. Thus, to prove Theorem \ref{thm compact of maximal singular integral 1}, it suffices to verify that the set $F_1:=\{T_{*,b,1}(f,g): f\in B_1,\,g\in B_2\}$ is strongly pre-compact in $L^r(\R^n)$ for $b\in C_c^\infty(\R^n)$. By Lemma \ref{Frechet Kolmogolov thm}, we need only to prove  $(a)$, $(b)$ and $(c)$ hold uniformly in $F_1$.

For $(a)$, by (\ref{bdd of T_{*,b,i}}), we easily obtain
\begin{align}\label{a1 one commutator}
\sup_{f\in B_1,g\in B_2}\|T_{*,b,1}(f,g)\|_r\leq C\|b\|_{BMO} \sup_{f\in B_1,g\in B_2}\|f\|_p\|g\|_q<C<\infty.
\end{align}
Notice $b\in C_c^\infty(\R^n)$, without loss of generality, we can assume that $\supp b\subset\{x\in\R^n:|x|\leq \beta\}$ with $\beta>1$. For any $\epsilon>0$, $0<s<\frac{n}{q}$, we take $\alpha>2\beta$ such that $\alpha^{-n-s+n/r}<\epsilon$, then we have
\begin{align}\label{b1 one commutator}
\|T_{*,b,1}(f,g)\chi_{E_\alpha}\|_r<C\epsilon.
\end{align}
In fact, combine with the support set of $b$ and notice that $K\in2$-$CZK(A,\gamma)$, for $|x|>\alpha$, H\"{o}lder's inequality gives
\begin{align*}
T_{*,b,1}(f,g)(x)\leq& C\int_{|y|\leq\beta}  \frac{|b(y)|}{|x-y|^{n+s}}|f(y)|\int_{\R^n}\frac{|g(z)|}{(|x-y|+|x-z|)^{n-s}}\;dzdy\\
\leq&C\|b\|_\infty\int_{|y|\leq\beta}  \frac{|f(y)|}{|x-y|^{n+s}}\;dy\bigg(\int_{\R^n}\frac{1}{(1+|x-z|)^{(n-s)q^\prime}}\;dz\bigg)^{1/q^\prime}\|g\|_q\\
\leq&C \beta^{\frac{n}{p'}}|x|^{-n-s} \|f\|_p\|g\|_q\leq C|x|^{-n-s}.
\end{align*}
Thus, we have
\begin{align*}
\bigg(\int_{\R^n}|T_{*,b,1}(f,g)(x)|^r\chi_{E_\alpha}(x)\;dx \bigg)^{1/r}\leq C\bigg(\int_{|x|>\alpha}|x|^{-(n+s)r}\;dx\bigg)^{\frac{1}{r}}
\leq C\epsilon.
\end{align*}
That is, (\ref{b1 one commutator}) holds uniformly for $f\in B_1,g\in B_2$.

It remains to prove that $(c)$ holds also for $T_{*,b,1}(f,g)$ uniformly with $f\in B_1,g\in B_2$. That is, we need to verify that for any $0<\epsilon<\frac{1}{4}$, if $|h|$ is sufficiently small and dependent only on $\epsilon$, then
\begin{align}\label{c1-1}
\|T_{*,b,1}(f,g)(\cdot+h)-T_{*,b,1}(f,g)(\cdot)\|_r<C\epsilon
\end{align} holds uniformly for $f\in B_1,g\in B_2$.

In fact, for any $h\in \R^n$, we denote $\widetilde{K}_\delta(x,y,z)=K(x,y,z)\chi_{|x-y|+|x-z|>\delta}$, then
\begin{align}\label{decomposition of T*b1}
|&T_{*,b,1}(f,g)(x+h)-T_{*,b,1}(f,g)(x)|\\
\leq&\sup_{\delta>0}\bigg|\int_{\R^n}\int_{\R^n}\widetilde{K}_\delta(x+h,y,z)(b(y)-b(x+h))f(y)g(z)dydz\nonumber\\
&-\int_{\R^n}\int_{\R^n}\widetilde{K}_\delta(x,y,z)(b(y)-b(x))f(y)g(z)dydz\bigg|.\nonumber
\end{align}
We can control the right hand side of the above inequality by the sum of the following four terms:
\begin{align*}
J_1:=\sup_{\delta>0}\bigg|\iint_{|x-y|+|x-z|>\epsilon^{-1}|h|}\widetilde{K}_\delta(x,y,z)(b(x)-b(x+h))f(y)g(z)dydz\bigg|,
\end{align*}
\begin{align*}
J_2:=\sup_{\delta>0}\bigg|\iint_{|x-y|+|x-z|>\epsilon^{-1}|h|}(\widetilde{K}_\delta(x+h,y,z)-\widetilde{K}_\delta(x,y,z))(b(y)-b(x+h))f(y)g(z)dydz\bigg|,
\end{align*}
\begin{align*}
J_3:=\sup_{\delta>0}\bigg|\iint_{|x-y|+|x-z|\leq\epsilon^{-1}|h|}\widetilde{K}_\delta(x,y,z)(b(y)-b(x))f(y)g(z)dydz\bigg|,
\end{align*}
\begin{align*}
J_4:=\sup_{\delta>0}\bigg|\iint_{|x-y|+|x-z|\leq\epsilon^{-1}|h|}\widetilde{K}_\delta(x+h,y,z)(b(y)-b(x+h))f(y)g(z)dydz\bigg|.
\end{align*}
We will give the estimates for $J_1,\,J_2,\,J_3,\,J_4$, respectively in the following.\\
\text{\it\bf {Estimate for $J_1$.}} It is easy to see that
\begin{align*}
J_1\leq &|b(x+h)-b(x)|\sup_{\delta>0}\bigg|\iint_{|x-y|+|x-z|>\epsilon^{-1}|h|} \widetilde{K}_\delta(x,y,z)f(y)g(z) \;dydz\bigg|\\
\leq&|h|\|\nabla b\|_\infty \sup_{\delta>0}
\bigg|\iint_{\gfz{|x-y|+|x-z|> \epsilon^{-1}|h|}{|x-y|+|x-z|> \delta}}
 K(x,y,z)f(y)g(z) \;dydz\bigg|\\
\leq&C|h|T_*(f,g)(x).
\end{align*}
Applying (\ref{bdd of T_*}), we obtain
\begin{align}\label{J_1}
\|J_1\|_r\leq C|h|.
\end{align}
\text{\it\bf {Estimate for $J_2$.}} Notice that
\begin{align*}
J_2\leq&\sup_{\delta>0}\bigg|\iint_{|x-y|+|x-z|>\epsilon^{-1}|h|}(K(x+h,y,z)-K(x,y,z))\chi_{|x+h-y|+|x+h-z|>\delta}\\
&\times(b(y)-b(x+h))f(y)g(z)dydz\bigg|\\
&+\sup_{\delta>0}\bigg|\iint_{|x-y|+|x-z|>\epsilon^{-1}|h|}K(x,y,z)(\chi_{|x+h-y|+|x+h-z|>\delta}-\chi_{|x-y|+|x-z|>\delta})\\
&\times(b(y)-b(x+h))f(y)g(z)dydz\bigg|\\
=&:J_{21}+J_{22}.
\end{align*}
Observe that if $|x-y|+|x-z|>\epsilon^{-1}|h|$ and $0<\epsilon<1/4$, then $|h|\leq \frac{1}{2}\max\{|x-y|,|x-z|,|y-z|\}$. Since $K\in 2$-$CZK(A,\gamma)$, we have
\begin{align*}
J_{21}\leq C\|b\|_\infty\iint_{|x-y|+|x-z|>\epsilon^{-1}|h|} \frac{|h|^\gamma}{(|x-y|+|x-z|)^{2n+\gamma}} |f(y)g(z)|\;dydz.
\end{align*}
Minkowski's inequality and H\"{o}lder's inequality give that
\begin{align*}
\|J_{21}\|_r&\leq C\iint_{|y|+|z|>\epsilon^{-1}|h|} \frac{|h|^\gamma}{(|y|+|z|)^{2n+\gamma}}\left(\int_{\R^n}|f(x-y)g(x-z)|^rdx\right)^{\frac{1}{r}} dydz\leq C\epsilon^\gamma.
\end{align*}
For $J_{22}$, it is easy to see that
\begin{align*}
J_{22}\leq&\sup_{\delta>0}\bigg|\iint_{\gfz{|x-y|+|x-z|>\epsilon^{-1}|h|}
{\gfz {|x+h-y|+|x+h-z|>\delta}{|x-y|+|x-z|\leq\delta}}}
K(x,y,z)(b(y)-b(x+h))f(y)g(z)dydz\bigg|\\
&+\sup_{\delta>0}\bigg|\iint_{\gfz{|x-y|+|x-z|>\epsilon^{-1}|h|}
{\gfz {|x+h-y|+|x+h-z|\leq\delta}{|x-y|+|x-z|>\delta}}}
K(x,y,z)(b(y)-b(x+h))f(y)g(z)dydz\bigg|\\
=&:J_{221}+J_{222}.
\end{align*}
For $J_{221}$, as $|x-y|+|x-z|>\epsilon^{-1}|h|$, $|x+h-y|+|x+h-z|>\delta$ and $0<\epsilon<1/4$, then $|x-y|+|x-z|\geq \frac{1}{2\epsilon+1}(|x+h-y|+|x+h-z|)\geq \frac{\delta}{2\epsilon+1}$. Then for any $1<r_0<\min\{p,q\}$, H\"{o}lder's inequality and (\ref{estimate 1}) give that
\begin{align*}
J_{221}\leq& C\|b\|_\infty\sup_{\delta>0}\iint_{\frac{\delta}{2\epsilon+1}\leq |x-y|+|x-z|\leq \delta}\frac{|f(y)g(z)|}{(|x-y|+|x-z|)^{2n}}dydz\\
\leq& C\sup_{\delta>0}\bigg(\iint_{\frac{\delta}{2\epsilon+1}\leq |y|+|z|\leq \delta}\frac{|f(x-y)g(x-z)|^{r_0}}{(|y|+|z|)^{2n}}\;dydz\bigg)^{\frac{1}{r_0}}\\
&\times\sup_{\delta>0}\bigg(\iint_{\frac{\delta}{2\epsilon+1}\leq |y|+|z|\leq \delta}\frac{dydz}{(|y|+|z|)^{2n}}\bigg)^{\frac{1}{r_0'}}\\
\leq& C\sup_{\delta>0}\bigg((2\epsilon+1)^{2n}\delta^{-2n}\iint_{|y|+|z|\leq \delta}|f(x-y)g(x-z)|^{r_0}\;dydz\bigg)^{\frac{1}{r_0}}[1-(2\epsilon+1)^{-n}]^{\frac{1}{r_0'}}\\
\leq&C\epsilon^{\frac{1}{r_0'}}M(|f|^{r_0})(x)^{\frac{1}{r_0}}M(|g|^{r_0})(x)^{\frac{1}{r_0}}.
\end{align*}Notice that $\frac{1}{r}=\frac{1}{p}+\frac{1}{q}$, and $\frac{p}{r_0},\frac{q}{r_0}>1$. Thus, H\"{o}lder's inequality and the $L^p$-boundedness of Hardy-Littlewood maximal operator $M$ give
\begin{align*}
\|J_{221}\|_r\leq& C\epsilon^{\frac{1}{r_0'}}\|M(|f|^{r_0})^{\frac{1}{r_0}}M(|g|^{r_0})^{\frac{1}{r_0}}\|_r\leq C\epsilon^{\frac{1}{r_0'}}\|M(|f|^{r_0})^{\frac{1}{r_0}}\|_p\|M(|g|^{r_0})^{\frac{1}{r_0}}\|_q\\
\leq & C\epsilon^{\frac{1}{r_0'}}\||f|^{r_0}\|_{p/r_0}^{\frac{1}{r_0}}\||g|^{r_0}\|_{q/r_0}^{\frac{1}{r_0}}\leq C\epsilon.
\end{align*}
The estimate for $J_{222}$ is completely similar. It is easy to see that as $|x-y|+|x-z|>\epsilon^{-1}|h|$, $|x+h-y|+|x+h-z|\leq\delta$ and $0<\epsilon<1/4$, then $|x-y|+|x-z|\leq \frac{1}{1-2\epsilon}(|x+h-y|+|x+h-z|)\leq \frac{\delta}{1-2\epsilon}$. Therefore, analogous to $J_{221}$, using H\"{o}lder's inequality and (\ref{estimate 2}), then for any $1<r_0<\min\{p,q\}$,
\begin{align*}
J_{222}\leq& C\|b\|_\infty\sup_{\delta>0}\iint_{\delta \leq |x-y|+|x-z|\leq \frac{\delta}{1-2\epsilon}}\frac{|f(y)g(z)|}{(|x-y|+|x-z|)^{2n}}\;dydz\\
\leq& CM(|f|^{r_0})(x)^{\frac{1}{r_0}}M(|g|^{r_0})(x)^{\frac{1}{r_0}} [(1-2\epsilon)^{-n}-1]^{\frac{1}{r_0'}}\\
\leq&C\epsilon^{\frac{1}{r_0'}}M(|f|^{r_0})(x)^{\frac{1}{r_0}}M(|g|^{r_0})(x)^{\frac{1}{r_0}}.
\end{align*}
Thus,  we also can obtain $$\|J_{222}\|_r\leq C\epsilon.$$
Combine with the estimates for $J_{21}$, $J_{221}$ and $J_{222}$, then
\begin{align}\label{J_2}
\|J_2\|_r\leq C\epsilon.
\end{align}
\text{\it\bf {Estimate for $J_3$.}} Note that, $|b(x)-b(y)|\leq \|\nabla b\|_\infty|x-y|$ and $K\in 2$-$CZK(A,\gamma)$, we have
\begin{align*}
J_3\leq C\|\nabla b\|_\infty\iint_{|x-y|+|x-z|\leq \epsilon^{-1}|h|} \frac{|f(y)g(z)|}{{(|x-y|+|x-z|)^{2n-1}}}\;dydz.
\end{align*}
Therefore, Minkowski's inequality and H\"{o}lder's inequality give that
\begin{align}\label{J_3}
\|J_3\|_r\leq C\iint_{|y|+|z|\leq\epsilon^{-1}|h|} \left(\int_{\R^n}|f(x-y)g(x-z)|^r\;dx\right)^{\frac{1}{r}}\frac{dydz}{(|y|+|z|)^{2n-1}} \leq C\epsilon^{-1}|h|.
\end{align}
\text{\it\bf {Estimate for $J_4$.}} With the same way, we have
\begin{align}\label{J_4}
\|J_4\|_r\leq C(2+\epsilon^{-1})|h|.
\end{align}
Note that the constants $C$ in (\ref{J_1})-(\ref{J_4}) are independent of $h$ and $\epsilon$. Taking $|h|$ to be sufficiently small, we obtain (\ref{c1-1}). Therefore, \text{(c)} holds for $T_{*,b,1}(f,g)$ uniformly for $f\in B_1,g\in B_2$. We complete the proof.

\section{The proof of Theorem \ref{thm compact of maximal singular integral 2}} 
\setcounter{equation}{0}

\noindent
\text{\it {Proof of Theorem \ref{thm compact of maximal singular integral 2}.}}
$T_{\ast,(b_1,b_2)}$ is bounded from $L^p(\R^n)\times L^q(\R^n)$ to $L^r(\R^n)$ by (\ref{bdd of T_{*,(b_1,b_2)}}). Let $B_1$, $B_2$ be unit balls in $L^p(\R^n)$ and $L^q(\R^n)$, respectively. We need to prove that the set $G=\{T_{\ast,(b_1,b_2)}(f,g):f\in B_1,g\in B_2\}$ is strongly pre-compact in $L^r(\R^n)$. Notice that
\begin{align*}\label{approximate VMO}
|&T_{*,(b_1,b_2)}(f,g)(x)-T_{*,(b_1^\epsilon,b_2^\epsilon)}(f,g)(x)|\\
\leq&\sup_{\delta>0}\bigg|\iint_{|x-y|+|x-z|>\delta} \bigg[(b_1(y)-b_1(x))(b_2(z)-b_2(x))-(b_1^\epsilon(y)-b_1^\epsilon(x))(b_2^\epsilon(z)-b_2^\epsilon(x)) \bigg]\nonumber\\
&\times K(x,y,z)f(y)g(z)\;dydz\bigg|\nonumber\\
\leq &T_{*,(b_1-b_1^\epsilon,b_2)}(f,g)(x)+T_{*,(b_1^\epsilon,b_2-b_2^\epsilon)}(f,g)(x).\nonumber
\end{align*} Therefore, if $b_j^\epsilon\in C_c^\infty(\R^n)$, such that $\|b_j-b_j^\epsilon\|_{BMO}<\epsilon$ ($j=1,2$), then (\ref{bdd of T_{*,(b_1,b_2)}}) gives that
\begin{align}
\|&T_{*,(b_1,b_2)}(f,g)-T_{*,(b_1^\epsilon,b_2^\epsilon)}(f,g)\|_r\leq \|T_{*,(b_1-b_1^\epsilon,b_2)}(f,g)\|_r+\|T_{*,(b_1^\epsilon,b_2-b_2^\epsilon)}(f,g)\|_r\\
\leq & C(\|b_2\|_{BMO}\|b_1-b_1^\epsilon\|_{BMO}+\|b_1^\epsilon\|_{BMO}\|b_2-b_2^\epsilon\|_{BMO})\|f\|_p\|g\|_q\leq C\epsilon.\nonumber
\end{align}Therefore, we only need to prove that $G$ is strongly pre-compact in $L^r(\R^n)$ for $b_j\in C_c^\infty(\R^n)$.
By Lemma \ref{Frechet Kolmogolov thm}, we only need to show that $(a)$, $(b)$ and $(c)$ in Lemma \ref{Frechet Kolmogolov thm} hold for $T_{*,(b_1,b_2)}(f,g)$ uniformly in $G$ with $b_1,b_2\in C_c^\infty(\R^n)$.

By (\ref{bdd of T_{*,(b_1,b_2)}}), we easily obtain that $(a)$ holds for $T_{*,(b_1,b_2)}(f,g)$ uniformly with $f\in B_1,g\in B_2$.  Notice $b_j\in C_c^\infty(\R^n)$, without loss of generality, we can also assume that $\supp b_j\subset\{x\in\R^n:|x|\leq \beta\}$ with $\beta>1$. For any $\epsilon>0$, we take $\alpha>2\beta$ such that $(\alpha-\beta)^{-2n+n/r}<\epsilon$, then we have
\begin{align}\label{b1}
\bigg(&\int_{\R^n}|T_{*,(b_1,b_2)}(f,g)(x)|^r\chi_{E_\alpha}(x)\;dx \bigg)^{\frac{1}{r}}\\
\leq& C\bigg(\int_{E_\alpha}\bigg(\int_{B(0,\beta)}\int_{B(0,\beta)} \frac{|b_1(y)||b_2(z)|}{(|x-y|+|x-z|)^{2n}}|f(y)||g(z)|\;dydz\bigg)^r\;dx \bigg)^{\frac{1}{r}}\nonumber\\
\leq& C\|b_1\|_\infty\|b_2\|_\infty\int_{B(0,\beta)}\int_{B(0,\beta)} \bigg(\int_{|x|>\alpha-\beta} \frac{dx}{|x|^{2nr}}\bigg)^{1/r}|f(y)||g(z)|\;dydz\nonumber\\
\leq& C\|b_1\|_\infty\|b_2\|_\infty\beta^{n(2-1/r)}\epsilon,\nonumber
\end{align}
which shows that $(b)$ holds for $T_{*,(b_1,b_2)}(f,g)$ uniformly with $f\in B_1,g\in B_2$. It remains to prove that for any $0<\epsilon<1/4$, if $|h|$ is sufficiently small and dependent only on $\epsilon$, then
\begin{align}\label{c1}
\|T_{*,(b_1,b_2)}(f,g)(\cdot+h)-T_{*,(b_1,b_2)}(f,g)(\cdot)\|_r<C\epsilon
\end{align}holds uniformly for $f\in B_1,g\in B_2$.

In fact, for any $h\in \R^n$, we denote $\vec{b}(x,y,z)=(b_1(y)-b_1(x))(b_2(z)-b_2(x))$ and $\vec{b}_h(x,y,z)=\vec{b}(x+h,y,z)-\vec{b}(x,y,z)$, using the notation in the proof of Theorem \ref{thm compact of maximal singular integral 1}, we have
\begin{align*}
|&T_{*,(b_1,b_2)}(f,g)(x+h)-T_{*,(b_1,b_2)}(f,g)(x)|\\
\leq&\sup_{\delta>0}\bigg|\int_{\R^n}\int_{\R^n}\widetilde{K}_\delta(x+h,y,z)\vec{b}(x+h,y,z) f(y)g(z)dydz\\
&-\int_{\R^n}\int_{\R^n}\widetilde{K}_\delta(x,y,z)\vec{b}(x,y,z)f(y)g(z)dydz\bigg|.
\end{align*}Similar to the decomposition for (\ref{decomposition of T*b1}), we can control the right hand side of the above inequality by $L_1+L_2+L_3+L_4$, where
\begin{align*}
L_1:=\sup_{\delta>0}\bigg|\iint_{|x-y|+|x-z|>\epsilon^{-1}|h|}\widetilde{K}_\delta(x,y,z)\vec{b}_h(x,y,z)f(y)g(z)dydz\bigg|,
\end{align*}
\begin{align*}
L_2:=&\sup_{\delta>0}\bigg|\iint_{|x-y|+|x-z|>\epsilon^{-1}|h|}(\widetilde{K}_\delta(x+h,y,z)-\widetilde{K}_\delta(x,y,z))\vec{b}(x+h,y,z)f(y)g(z)dydz\bigg|,
\end{align*}
\begin{align*}
L_3:=\sup_{\delta>0}\bigg|\iint_{|x-y|+|x-z|\leq\epsilon^{-1}|h|}\widetilde{K}_\delta(x,y,z)\vec{b}(x,y,z)f(y)g(z)dydz\bigg|,
\end{align*}
\begin{align*}
L_4:=&\sup_{\delta>0}\bigg|\iint_{|x-y|+|x-z|\leq\epsilon^{-1}|h|}\widetilde{K}_\delta(x+h,y,z)\vec{b}(x+h,y,z)f(y)g(z)dydz\bigg|.
\end{align*}
In the following, we will give the estimates for $L_1,\,L_2,\,L_3,\,L_4$, respectively.\\
\text{\it\bf {Estimate for $L_1$.}}
Observe
\begin{align}\label{decomposition of b}
&\vec{b}_h(x,y,z)
=(b_1(x)-b_1(x+h))(b_2(x)-b_2(x+h))\\
+&(b_1(x)-b_1(x+h))(b_2(z)-b_2(x))+(b_1(y)-b_1(x))(b_2(x)-b_2(x+h))\nonumber.
\end{align}
Note that $|b_j(x)-b_j(x+h)|\leq |h|\|\nabla b_j\|_\infty$, then we have
\begin{align*}
L_1\leq &|h|^2\|\nabla b_1\|_\infty\|\nabla b_2\|_\infty\sup_{\delta>0}\bigg|\iint_{|x-y|+|x-z|> \epsilon^{-1}|h|} \widetilde{K}_\delta(x,y,z) f(y)g(z)\;dydz\bigg|\nonumber\\
&+|h|\|\nabla b_1\|_\infty\sup_{\delta>0}\bigg|\iint_{|x-y|+|x-z|> \epsilon^{-1}|h|} \widetilde{K}_\delta(x,y,z) (b_2(z)-b_2(x)) f(y)g(z)\;dydz\bigg|\\
&+|h|\|\nabla b_2\|_\infty\sup_{\delta>0}\bigg|\iint_{|x-y|+|x-z|> \epsilon^{-1}|h|} \widetilde{K}_\delta(x,y,z) (b_1(y)-b_1(x)) f(y)g(z)\;dydz\bigg|\\
\leq& C|h|^2 T_\ast(f,g)(x)+C|h|T_{*,b_2,2}(f,g)(x)+C|h|T_{*,b_1,1}(f,g)(x).
\end{align*}
Then (\ref{bdd of T_*}) and (\ref{bdd of T_{*,b,i}}) give that
\begin{align}\label{L_1}
\|L_1\|_r\leq C(|h|^2+|h|\|b_1\|_{BMO}+|h|\|b_2\|_{BMO})\|f\|_p\|g\|_q\leq C(|h|+|h|^2).
\end{align}
\text{\it\bf {Estimate for $L_2$.}} The estimate for $L_2$ is similar to $J_2$. We have
\begin{align*}
L_2\leq&\sup_{\delta>0}\bigg|\iint_{|x-y|+|x-z|>\epsilon^{-1}|h|}(K(x+h,y,z)-K(x,y,z))\chi_{|x+h-y|+|x+h-z|>\delta}\\
&\times\vec{b}(x+h,y,z)f(y)g(z)dydz\bigg|\\
+&\sup_{\delta>0}\bigg|\iint_{\gfz{|x-y|+|x-z|>\epsilon^{-1}|h|}
{\gfz {|x+h-y|+|x+h-z|>\delta}{|x-y|+|x-z|\leq\delta}}}
K(x,y,z)\vec{b}(x+h,y,z)f(y)g(z)dydz\bigg|\\
+&\sup_{\delta>0}\bigg|\iint_{\gfz{|x-y|+|x-z|>\epsilon^{-1}|h|}
{\gfz {|x+h-y|+|x+h-z|\leq\delta}{|x-y|+|x-z|>\delta}}}
K(x,y,z)\vec{b}(x+h,y,z)f(y)g(z)dydz\bigg|\\
=&:L_{21}+L_{22}+L_{23}.
\end{align*}The estimates for $L_{21},\,L_{22},\,L_{23}$ are completely analogous to $J_{21},\,J_{221},\,J_{222}$. Then,
\begin{align}\label{L_2}
\|L_2\|_r\leq C\epsilon.
\end{align}
\text{\it\bf {Estimate for $L_3$.}}
Note that, $|b_j(x)-b_j(y)|\leq \|\nabla b_j\|_\infty|x-y|$ $(j=1,2)$ and $K\in 2$-$CZK(A,\gamma)$, we have
\begin{align*}
L_3\leq C\|\nabla b_1\|_\infty\|\nabla b_2\|_\infty\iint_{|x-y|+|x-z|\leq \epsilon^{-1}|h|} \frac{ |f(y)g(z)|}{{(|x-y|+|x-z|)^{2n-2}}}\;dydz.
\end{align*}
Therefore, by Minkowski's inequality and H\"{o}lder's inequality, we have
\begin{align}\label{L_3}
\|L_3\|_r\leq &C\|\nabla b_1\|_\infty\|\nabla b_2\|_\infty\iint_{|y|+|z|\leq \epsilon^{-1}|h|} \bigg(\int_{\R^n}|f(x-y)g(x-z)|^rdx\bigg)^{\frac{1}{r}}\frac{dydz}{{(|y|+|z|)^{2n-2}}}\\ \leq& C(\epsilon^{-1}|h|)^2\|f\|_p\|g\|_q\leq C(\epsilon^{-1}|h|)^2.\nonumber
\end{align}
\text{\it\bf {Estimate for $M_4$.}} With the same way, we have
\begin{align}\label{L_4}
\|L_4\|_r\leq C(2+\epsilon^{-1})^2|h|^2.
\end{align}
Note that the constants $C$ in (\ref{L_1})-(\ref{L_4}) are independent of $h$ and $\epsilon$. Taking $|h|<(2+\epsilon^{-1})^{-1}\epsilon$, we obtain (\ref{c1}). Therefore, $(c)$ holds for $T_{*,(b_1,b_2)}(f,g)$ uniformly for $f\in B_1,g\in B_2$. We complete the proof.

\end{document}